\documentclass[amscd,amssymb,verbatim,11pt]{amsart}
\usepackage{epsfig}

\newtheorem{theorem}{Theorem}[section]

\begin{document}

\title{Generalized Bernoulli numbers and  a formula of Lucas}

\author{Victor H. Moll}
\address{Department of Mathematics,
Tulane University, New Orleans, LA 70118}
\email{vhm@tulane.edu}

\author{Christophe Vignat}
\address{Department of Mathematics,
Tulane University, New Orleans, LA 70118}
\email{cvignat@math.tulane.edu}


\date{\today}


\begin{abstract}
An overlooked formula of E.~Lucas for the generalized Bernoulli numbers is proved using generating 
functions. This is then used to provide a new proof and a new form of a sum involving 
classical Bernoulli numbers studied by 
K.~Dilcher.  The value of this sum is then given in terms of the  Meixner-Pollaczek polynomials. 
\end{abstract}

\maketitle

\newcommand{\nn}{\nonumber}
\newcommand{\ba}{\begin{eqnarray}}
\newcommand{\ea}{\end{eqnarray}}
\newcommand{\no}{\noindent}
\newcommand{\realpart}{\mathop{\rm Re}\nolimits}
\newcommand{\imagpart}{\mathop{\rm Im}\nolimits}
\newcommand{\pFq}[5]{\ensuremath{{}_{#1}F_{#2} \left( \genfrac{}{}{0pt}{}{#3}
{#4} \bigg| {#5} \right)}}

\newtheorem{Definition}{\bf Definition}[section]
\newtheorem{Thm}[Definition]{\bf Theorem}
\newtheorem{Example}[Definition]{\bf Example}
\newtheorem{Lem}[Definition]{\bf Lemma}
\newtheorem{Note}[Definition]{\bf Note}
\newtheorem{Cor}[Definition]{\bf Corollary}
\newtheorem{Conj}[Definition]{\bf Conjecture}
\newtheorem{Prop}[Definition]{\bf Proposition}
\newtheorem{Problem}[Definition]{\bf Problem}
\numberwithin{equation}{section}

\section{Introduction} \label{sec-intro}
\setcounter{equation}{0}

The goal of this paper is to provide a unified approach to two topics 
that have appeared in the literature. The first one is an expression for 
the generalized Bernoulli numbers $B_{n}^{(p)}$ defined by the 
exponential generating function 
\begin{equation}
\sum_{n=0}^{\infty} B_{n}^{(p)} \frac{z^{n}}{n!} = \left( \frac{z}{e^{z}-1} \right)^{p}.
\label{gen-ber}
\end{equation}
\noindent
For $n \in \mathbb{N}$, the coefficients $B_{n}^{(p)}$ are polynomials in $p$ named after 
 N\"{o}rlund in \cite{carlitz-1960a}.  The first few are 
 \begin{equation}
 B_{0}^{(p)} = 1, \, B_{1}^{(p)} = -\tfrac{1}{2}p, \, B_{2}^{(p)} = -\tfrac{1}{12}p + 
 \tfrac{1}{4}p^{2}, \, B_{3}^{(p)} = \tfrac{1}{8}p^{2}(1-p).
 \end{equation}
 \noindent
 In his $1878$ paper E.~Lucas \cite{lucase-1878a} gave the  formula 
 \begin{equation}
 B_{n}^{(p)} = \frac{(-1)^{p-1}}{(p-1)!} \frac{n!}{(n-p)!} \beta^{n-p+1} 
 (1+\beta) \cdots (p-1+\beta)
 \label{lucas-orig}
 \end{equation}
 \noindent
 for $n \geq p$. This is a symbolic formula: to obtain the 
 value of $B_{n}^{(p)}$, expand the expression \eqref{lucas-orig} 
 and replace $\beta^{j}$ by the ratio $B_{j}/j$. Here $B_{j}$ is the classical Bernoulli number 
 $B_{n} = B_{n}^{(1)}$ in the notation from \eqref{gen-ber}.
 
The second topic is an expression established by 
K. Dilcher \cite{dilcher-1996a} for the sums of products of Bernoulli numbers
\begin{equation}
S_{N}(n) := \sum \binom{2n}{2j_{1}, \, 2j_{2}, \, \cdots, 2j_{N}} 
B_{2j_{1}}B_{2j_{2}} \cdots B_{2j_{N}},
\label{dilcher-sum}
\end{equation}
\noindent
where the sum is taken over all nonnegative integers $j_{1}, \cdots, j_{N}$ 
such that $j_{1}+ \cdots + j_{N} = n$, and where 
\begin{equation}
\binom{2n}{2j_{1}, \, 2j_{2}, \, \cdots, 2j_{N}}  = 
\frac{(2n)!}{(2j_{1})! \cdots (2j_{N})!} 
\end{equation}
\noindent
is the multinomial coefficient and $B_{2k}$ is the classical Bernoulli number.
One of the  main results of \cite{dilcher-1996a} is the evaluation
\begin{equation}
S_{N}(n) = \frac{(2n)!}{(2n-N)!} 
\sum_{k=0}^{\lfloor{ (N-1)/2 \rfloor}} b_{k}^{(N)} \frac{B_{2n-2k}}{2n-2k},
\label{dilcher-coeff}
\end{equation}
\noindent
where the coefficients $b_{k}^{(N)}$ are defined by the recurrence 
\begin{equation}
b_{k}^{(N+1)} = - \frac{1}{N} b_{k}^{(N)} + \frac{1}{4} b_{k-1}^{(N-1)},
\label{dilcher-rec}
\end{equation}
\noindent
with $b_{0}^{(1)} = 1$ and $b_{k}^{(N)} = 0$ for $k  < 0$ and for 
$k > \lfloor (N-1)/2 \rfloor$. 

Lucas's  original proof is recalled in Section \ref{sec-lucas}.  This section 
also contains an extension of Lucas's formula for $B_{n}^{(p)}$ to $0 \leq n \leq p-1$ in terms 
of the Stirling numbers of the first kind. A unified proof of the two formulas 
for $B_{n}^{(p)}$ based on generating functions is given in Section \ref{sec-generating}. 
Another proof of Lucas's formula, based on recurrences, is given in Section \ref{sec-recurrences}
and Section \ref{sec-dilcher} contains a proof of 
\begin{equation}
S_{N}(n) = \sum_{k=0}^{N} \frac{(2n)!}{(2n-k)!} 2^{-k} \binom{N}{k} B_{2n-k}^{(N-k)}
\end{equation}
\noindent
that expresses Dilcher's sum \eqref{dilcher-sum} explicitly in terms of the generalized 
Bernoulli numbers. Expressing this result in hypergeometric form leads to a formula for 
$S_{N}(n)$ in terms of the Meixner-Pollaczek polynomials 
\begin{equation}
P_{n}^{(\lambda)}(x;\phi) = \frac{(2 \lambda)_{n}}{n!} e^{\imath n \phi} 
\pFq21{-n \quad \lambda + \imath x}{2 \lambda}{1 - e^{-2 \imath \phi}}. 
\end{equation}
\noindent
It is then established that the recurrence \eqref{dilcher-rec}, provided by Dilcher 
in \cite{dilcher-1996a}, is 
equivalent to the classical three-term relation for this orthogonal family of polynomials.

\section{Lucas's  theorem}
\label{sec-lucas}
\setcounter{equation}{0}

In his paper \cite{lucase-1878a}, E.~Lucas gave an expression for the generalized 
Bernoulli numbers $B_{n}^{(p)}$, for $n \geq p$. This section presents an outline of his proof and an 
extension of this 
expression for $B_{n}^{(p)}$ to the case $0 \leq n \leq p-1$. A proof based 
on generating functions is given in the next section.  Lucas's formula uses the translation 
\begin{equation}
\beta^{n} = \frac{B_{n}}{n}
\label{symb}
\end{equation}
\noindent 
coming from umbral calculus. Observe, for example,  that 
\begin{eqnarray*}
B_{3}^{(2)} & = & \frac{(-1)^{1}}{1!} \frac{3!}{1!} \beta^{2}(1+ \beta) = - 6(\beta^{2} + \beta^{3} ) \\
 & = & -6 \left( \frac{B_{2}}{2} + \frac{B_{3}}{3} \right) = -3B_{2} = - \frac{1}{2}
 \end{eqnarray*}
 \noindent
 Observe also that the symbolic substitution \eqref{symb} should be performed only \textit{after} all the 
 terms have been expanded. For example, 
 \begin{equation}
 \beta^{2}( 1+ \beta) = \beta^{2} + \beta^{3} = \frac{B_{2}}{2} + \frac{B_{3}}{3} = - \frac{1}{4}
 \end{equation}
 \noindent
 but 
 \begin{equation}
 \beta^{2}(1+ \beta) \neq \frac{B_{2}}{2} \left( 1 + \frac{B_{1}}{1} \right) = \frac{1}{24}.
 \end{equation}

\begin{theorem}[Lucas]
\label{thm-Lucas}
For $n \geq p$, the generalized Bernoulli numbers $B_{n}^{(p)}$ are given by 
\begin{equation}
B_{n}^{(p)} = \frac{(-1)^{p-1}}{(p-1)!} \frac{n!}{(n-p)!} \beta^{n-p+1} 
(1+\beta) (2 + \beta) \cdots (p-1+\beta)
\label{lucas-0}
\end{equation}
\noindent
where, in symbolic notation, 
\begin{equation}
\beta^{n} = \frac{B_{n}}{n}.
\end{equation}
\end{theorem}
\begin{proof}
Lucas's argument begins with the identity 
\begin{equation}
pB_{n}^{(p+1)} = (p-n) B_{n}^{(p)} - pn B_{n-1}^{(p)}
\label{lucas-1}
\end{equation}
\noindent
which follows directly from the identity for generating functions
\begin{equation}
x \frac{d}{dx} \left( \frac{x}{e^{x} - 1 } \right)^{p} = p (1-x) 
\left( \frac{x}{e^{x}-1} \right)^{p} - p 
\left( \frac{x}{e^{x}-1} \right)^{p+1}.
\end{equation}
\noindent
Shifting $n$ to $n-1$ it follows that
\begin{equation}
pB_{n-1}^{(p+1)} = (p-n+1) B_{n-1}^{(p)} - p(n-1) B_{n-2}^{(p)}.
\label{lucas-2}
\end{equation}
\noindent
Now multiplying  \eqref{lucas-1} by $n(p+1)$ and \eqref{lucas-2} by 
$(p-n+1)$ leads to 
\begin{eqnarray*}
p(p+1)B_{n}^{(p+2)} & = &  (p-n+1)(p-n)B_{n}^{(p)} - (p-n+1)(p+p+1)nB_{n-1}^{(p)} \\
& & + p(p+1)n(n-1)B_{n-2}^{(p)}
\end{eqnarray*}
\noindent
and then, by the same methods, he produces
\begin{eqnarray*}
(p+2)(p+1)pB_{n}^{(p+3)} & = & (p-n+2)(p-n+1)(p-n)B_{n}^{(p)} \\
& - & (p-n_2)(p-n+1)(p+p+1+p+2) nB_{n-1}^{(p)} \\
& + & (p-n+2)(p(p+1)+p(p+2) 
+  (p+1)(p+2))n(n-1)B_{n-2}^{(p)} \\
& - & p(p+1)(p+2)n(n-1)(n-2)B_{n-3}^{(p)}
\end{eqnarray*}
\noindent
and then, stating `\textit{and so on}', concludes the proof.
\end{proof}

The  following alternate proof of Lucas's  theorem using generating functions requires an expression 
for $B_{n}^{(p)}$ in the range $0 \leq n \leq p-1$, of the kind given in \eqref{lucas-0}. This cannot be obtained directly from 
\eqref{lucas-0}. The Stirling numbers of the first kind $s_{k}^{(p)}$ are used to produce an equivalent 
formulation of $B_{n}^{(p)}$. These numbers are defined  by the generating function
\begin{equation}
z(z-1)(z-2) \cdots (z-(p-1)) = \sum_{k=1}^{p} s_{k}^{(p)} z^{k}.
\label{gen-stirling}
\end{equation}
Then \eqref{lucas-0} may be written as 
\begin{eqnarray*}
B_{n}^{(p)} & = & \frac{(-1)^{p-1}}{(p-1)!} n(n-1)\cdots (n-(p-1)) \beta^{n-p} 
(-1)^{p} \sum_{k=1}^{p} s_{k}^{(p)} (- \beta)^{k} \\
& = & - \frac{1}{(p-1)!} n(n-1) \cdots (n-(p-1)) \sum_{k=1}^{p} s_{k}^{(p)} 
(-1)^{k} \frac{B_{n-p+k}}{n-p+k}.
\end{eqnarray*}
\noindent 
Observe that the index $n$ varies in the range $0 \leq n \leq p-1$, therefore 
the prefactor $n(n-1)\cdots (n-(p-1))$ always vanishes. On the other hand, 
all the summands are finite, except when $k = p-n$.  Performing the 
translation from $(-\beta)^{k}$ to $B_{k}/k$ for this specific index gives
\begin{equation*}
-\frac{1}{(p-1)!} n(n-1)\cdots 1 \times (-1)(-2) \cdots (-(p-1-n))) s_{p-n}^{(p)} (-1)^{p-n} =
\frac{s_{p-n}^{(p)}}{\binom{p-1}{n}}.
\end{equation*}

This gives:

\begin{theorem}
\label{the-stir1}
The generalized Bernoulli numbers $B_{n}^{(p)}$, with $0 \leq n \leq p-1$ are given by
\begin{equation}
B_{n}^{(p)} = \frac{s_{p-n}^{(p)}}{\binom{p-1}{n}}.
\label{ber-small}
\end{equation}
\end{theorem}

In fact, this is a classical result. It is, for example, a direct consequence of the identity
\begin{equation}
(z-1)(z-2) \cdots (z-p) = \sum_{\ell = 0}^{p} \binom{p}{\ell} z^{\ell} B_{p- \ell}^{(p+1)} 
\end{equation}
\noindent
which appears (unnumbered) in \cite[p.149]{norlund-1924a}.  

\section{The proof via generating function}
\label{sec-generating}
\setcounter{equation}{0}

The expressions for the generalized Bernoulli numbers given in \eqref{lucas-0} and \eqref{ber-small}  are now
used to compute the generating function
\begin{equation}
G(z) = \sum_{n=0}^{\infty} B_{n}^{(p)} \frac{z^{n}}{n!}
\end{equation}
\noindent
and to show that it coincides with the generating function of the generalized Bernoulli numbers \eqref{gen-ber}.

Split the sum as $G(z) = G_{1}(z) + G_{2}(z)$, where 
\begin{equation}
G_{1}(z) = \sum_{n=0}^{p-1} B_{n}^{(p)} \frac{z^{n}}{n!} \text{ and }
G_{2}(z) = \sum_{n=p}^{\infty} B_{n}^{(p)} \frac{z^{n}}{n!}.
\end{equation}

Observe that 
\begin{eqnarray*}
G_{2}(z) & = & \sum_{n=p}^{\infty} \frac{(-1)^{p-1}}{(p-1)!} \frac{n!}{(n-p)!} 
\beta^{n-p+1} (1+\beta) \cdots ( (p-1) + \beta) \frac{z^{n}}{n!} \\
& = & \frac{(-1)^{p-1}}{(p-1)!} \beta (1+ \beta) \cdots (p-1+ \beta) \sum_{n=p}^{\infty} 
\frac{n!}{(n-p)!} \beta^{n-p} \frac{z^{n}}{n!}  \\
& = & \frac{(-1)^{p-1}}{(p-1)!} (-1)^{p} \sum_{k=1}^{p} s_{k}^{(p)} (-1)^{k} z^{p} f_{k}(z) 
\end{eqnarray*}
\noindent
with 
\begin{equation}
f_{k}(z) = \sum_{n=p}^{\infty} \frac{B_{n-p+k}}{(n-p)! (n-p+k)} z^{n-p}.
\end{equation}
\noindent 
The $(k-1)$-st antiderivative of $f_{k}(z)$, denoted by $g_{k}(z)$, is 
\begin{eqnarray*}
g_{k}(z)  & = & \sum_{n=p}^{\infty} \frac{B_{n-p+k}}{(n-p+k)!} z^{n-p+k-1} \\
& = & z^{-1} \sum_{\ell = k}^{\infty} \frac{B_{\ell}}{\ell!} z^{\ell} \\
& = & \frac{1}{z} \left[  \frac{z}{e^{z} - 1 } - \sum_{\ell = 0}^{k-1} \frac{B_{\ell}}{\ell !} z^{\ell} \right],
\end{eqnarray*}
\noindent
therefore
\begin{eqnarray*}
f_{k}(z) & = & \left( \frac{d}{dz} \right)^{k-1} \frac{1}{e^{z}-1} - \left( \frac{d}{dz} \right)^{k-1} \frac{1}{z} \\
            & = & \left( \frac{d}{dz} \right)^{k-1} \frac{1}{e^{z}-1} + \frac{(-1)^{k} (k-1)!}{z^{k}}.
\end{eqnarray*}

This gives
\begin{eqnarray*}
G_{2}(z) & = & - \frac{z^{p}}{(p-1)!} \sum_{k=1}^{p} s_{k}^{(p)} (-1)^{k} f_{k}(z) \\
& = & - \frac{z^{p}}{(p-1)!} \sum_{k=1}^{p} s_{k}^{(p)} (-1)^{k} 
\left( \frac{d}{dz} \right)^{k-1} \left[ \frac{1}{e^{z}-1} \right] - 
\frac{z^{p}}{(p-1)!} \sum_{k=1}^{p} s_{k}^{(p)} \frac{(k-1)!}{z^{k}}.
\end{eqnarray*}

On the other hand,
\begin{eqnarray*}
G_{1}(z) & = & \sum_{n=0}^{p-1} B_{n}^{(p)} \frac{z^{n}}{n!} \\
& = & \sum_{n=0}^{p-1} \frac{s_{p-n}^{(p)}}{\binom{p-1}{n}} \frac{z^{n}}{n!} \\
& = & \frac{1}{(p-1)!} \sum_{n=0}^{p-1} s_{p-n}^{(p)} (p-1-n)! z^{n} \\
& = & \frac{1}{(p-1)!} \sum_{k=1}^{p} s_{k}^{(p)} (k-1)! z^{p-k}.
\end{eqnarray*}
\noindent
This sum cancels the second term in the expression for $G_{2}(z)$. Hence 
\begin{equation}
G(z) = G_{1}(z) + G_{2}(z) = - \frac{z^{p}}{(p-1)!} \sum_{k=1}^{p} s_{k}^{(p)} (-1)^{k} 
\left( \frac{d}{dz} \right)^{k-1} \left[ \frac{1}{e^{z} -1 } \right].
\end{equation}

Using \eqref{gen-stirling} this gives 
\begin{equation}
G(z) = - \frac{(-z)^{p}}{(p-1)!} \left( (p-1) + \frac{d}{dz} \right) \cdots \left( 1 + \frac{d}{dz} \right) 
\left[ \frac{1}{e^{z}-1} \right]. \label{form-G}
\end{equation}

The next lemma simplifies this expression. Its proof by induction is elementary, so it is omitted.

\begin{Lem}
For $n \geq 1$, the identity
\begin{equation}
\frac{(-1)^{n}}{n!} \left( n + \frac{d}{dz} \right) \left( n - 1 + \frac{d}{dz} \right) \cdots 
\left( 1 + \frac{d}{dz} \right) \frac{1}{e^{z}-1} = \frac{1}{(e^{z}-1)^{n+1}}
\end{equation}
\noindent
holds.
\end{Lem}

Replacing in \eqref{form-G} produces 
\begin{equation}
G(z) = - \frac{(-z)^{p}}{(p-1)!} \frac{(p-1)!}{(-1)^{p-1}} \frac{1}{(e^{z}-1)^{p}} = 
\left( \frac{z}{e^{z}-1} \right)^{p},
\end{equation}
\noindent
which is the generating function of the generalized Bernoulli numbers.  This proves both
Lucas's formula for $B_{n}^{(p)}$ with $n \geq p$ and the expression \eqref{ber-small} 
for $0 \leq p \leq n-1$. 

\section{Lucas's formula via recurrences}
\label{sec-recurrences}
\setcounter{equation}{0}

The numbers $B_{n}^{(p)}$ satisfy the recurrence 
\begin{equation}
pB_{n}^{(p+1)} = (p-n) B_{n}^{(p)} - pn B_{n-1}^{(p)}.
\label{recu-newber1}
\end{equation}
Lucas's  formula for $B_{n}^{(p)}$ is now established by showing that the 
numbers defined by 
\eqref{lucas-0} satisfy the same recurrence.

Start with
\begin{multline*}
(p-n)B_{n}^{(p)} - pnB_{n-1}^{(p)}  =  (p-n) \frac{(-1)^{p-1} n!}{(p-1)!(n-p)!} \beta^{n-p} 
\prod_{k=0}^{p-1} (k+ \beta) -  \\
p n \frac{(-1)^{p-1} n!}{(p-1)!(n-p-1)!} \beta^{n-1-p} 
\prod_{k=0}^{p-1} (k+ \beta),
\end{multline*}
\noindent 
and write it as 
\begin{eqnarray*}
(p-n)B_{n}^{(p)} - pnB_{n-1}^{(p)}  & = & 
\frac{(-1)^{p-1}n!}{(p-1)! (n-p-1)!} \beta^{n-1-p} 
\left[ - \prod_{k=0}^{p-1} (k + \beta) - p \beta \prod_{k=0}^{p-1} (k + \beta) \right] \\
& = & \frac{(-1)^{p}n!}{(p-1)! (n-p-1)!} \beta^{n-1-p} (p + \beta) \prod_{k=0}^{p-1} (k+ \beta) \\
& = & p \frac{(-1)^{p}}{p!} \frac{n!}{(n-p-1)!} \beta^{n-1-p} \prod_{k=0}^{p} (k + \beta) \\
& = & pB_{n}^{(p+1)}.
\end{eqnarray*}

To conclude the result, it suffices to check that the initial conditions match. This is 
clear, since
\begin{equation}
B_{n}^{(1)} = \frac{n!}{(n-1)!} \beta^{n} = n \beta^{n} = n \frac{B_{n}}{n} = B_{n}.
\end{equation}
This establishes  Lucas's  formula for the generalized Bernoulli numbers. 

\section{A new proof of Dilcher's formula}
\label{sec-dilcher}
\setcounter{equation}{0}

This section analyzes the sum 
\begin{equation}
S_{N}(n) := \sum \binom{2n}{2j_{1}, \, 2j_{2}, \, \cdots, 2j_{N}} 
B_{2j_{1}}B_{2j_{2}} \cdots B_{2j_{N}},
\end{equation}
\noindent
using Lucas's  expression for the generalized Bernoulli numbers $B_{n}^{(p)}$.  
An alternative formulation is presented.

\begin{Prop}
\label{propos-1}
The sum $S_{N}(n)$ is given by 
\begin{equation}
S_{N}(n) = \sum_{k=0}^{N} \frac{(2n)!}{(2n-k)!} 2^{-k} \binom{N}{k} B_{2n-k}^{(N-k)}
\end{equation}
\noindent
for $2n> N$.
\end{Prop}
\begin{proof}
The umbral method \cite{roman-1984a} shows that the sum $S_{N}(n)$ is given by
\begin{equation}
S_{N}(n) = \frac{1}{2^{N}} \left( \epsilon_{1} B_{1} + \cdots + \epsilon_{N} B_{N} \right)^{2n}
\end{equation}
\noindent 
with $\epsilon_{j} = \pm 1$. Introduce the notation
\begin{equation}
Y_{2n}^{(M,N)} = \left( - B_{1} - \cdots - B_{M} + B_{M+1} + \cdots + B_{N} \right)^{2n}
\end{equation}
\noindent
where there are $M$ minus signs and $N-M$ plus signs. Thus, 
\begin{equation}
S_{N}(n) = \frac{1}{2^{N}} \sum_{M=0}^{N} \binom{N}{M} Y_{2n}^{(M,N)}.
\label{sum-1}
\end{equation}

The next step uses the famous umbral identity 
\begin{equation}
f(-B) = f(B) + f'(0)
\end{equation}
(see Section 2 of \cite{dixit-2014a} for details) to obtain
\begin{equation}
Y_{2n}^{(M,N)} = Y_{2n}^{(M-1,N)} + 2n Y_{2n-1}^{(M-1,N-1)}.
\end{equation}
\noindent
This may  be written as 
\begin{equation}
Q_{2n}^{(M)} = Q_{2n}^{(M-1)} + 2n Q_{2n-1}^{(M-1)},
\label{easy-red}
\end{equation}
\noindent
where $Q_{j}^{M} = Y_{j}^{(M,P+j)}$ and $P= N-2n$.  Then \eqref{easy-red}
is easily solved to produce 
\begin{equation}
Q_{2n}^{(M)} = \sum_{k=0}^{M} \binom{M}{k} \frac{(2n)!}{(2n-k)!} Q_{2n-k}^{(0)}. 
\end{equation}
\noindent
Since the initial condition is 
\begin{equation}
Q_{2n-k}^{(0)} =  Y_{2n-k}^{(0, N-k)} = B_{2n-k}^{(N-k)},
\end{equation}
\noindent
it follows that 
\begin{equation}
Y_{2n}^{(M,N)} = \sum_{k=0}^{M} \binom{M}{k} \frac{(2n)!}{(2n-k)!} B_{2n-k}^{(N-k)}.
\end{equation}
Replacing in \eqref{sum-1} yields 
\begin{eqnarray*}
S_{N}(n) & = & \frac{1}{2^{N}} \sum_{M=0}^{N} \binom{N}{M} Y_{2n}^{(M,N)} \\
& = & \frac{1}{2^{N}} \sum_{M=0}^{N} \binom{N}{M} \sum_{k=0}^{M} \binom{M}{k} 
\frac{(2n)!}{(2n-k)!} B_{2n-k}^{(N-k)} \\
& = & \frac{1}{2^{N}} \sum_{k=0}^{N} \frac{(2n)!}{(2n-k)!} B_{2n-k}^{(N-k)} 
\sum_{M=0}^{N} \binom{M}{k} \binom{N}{M}.
\end{eqnarray*}
\noindent
Now use the basic identity
\begin{equation}
\sum_{M=0}^{N} \binom{M}{k} \binom{N}{M} = \sum_{M=k}^{N} \binom{M}{k} \binom{N}{M} =
2^{N-k} \binom{N}{k}
\end{equation}
\noindent
to obtain the result. 
\end{proof}

Lucas's  identity for generalized Bernoulli numbers is now used to 
obtain a  second expression for the sum $S_{N}(n)$. 

\begin{Prop}
\label{propos-2}
For $2n> N$, the sum $S_{N}(n)$ is given by 
\begin{equation}
S_{N}(n) = \frac{(2n)!}{(2n-N)!} \beta^{2n-N+1} 
\sum_{\ell=0}^{N-1} \binom{N}{\ell+1} \frac{(-1)^{\ell}}{2^{N-1-\ell}} \frac{(\beta+ 1)_{\ell}}{\ell!}.
\end{equation}
\end{Prop}
\begin{proof}
Using the Pochhammer symbol 
\begin{equation}
(\beta+1)_{p-1} = \frac{\Gamma(\beta+p)}{\Gamma(\beta+1)} = (\beta+1) \cdots (\beta +p-1)
\end{equation}
\noindent
Lucas's  formula \eqref{lucas-0} is stated in the form 
\begin{equation}
B_{n}^{(p)} = \frac{(-1)^{p-1}}{(p-1)!} \frac{n!}{(n-p)!} \beta^{n-p+1} ( \beta + 1)_{p-1}.
\end{equation}
\noindent
Using Proposition \ref{propos-1} and $B_{n}^{(0)} = \delta_{n}$ so that $B_{2n-N}^{(0)} = 0$ 
since $2n > N$, it follows that 
\begin{eqnarray*}
S_{N}(n) & = &  \sum_{k=0}^{N-1} \frac{(2n)!}{(2n-k)!} 2^{-k} \binom{N}{k} 
\frac{(-1)^{N-k-1}}{(N-k-1)!} \frac{(2n-k)!}{(2n-N)!} \beta^{2n-N + 1} (\beta+1)_{N-k-1} \\
& = & \frac{(2n)!}{(2n-N)!} \beta^{2n-N+1} \sum_{k=0}^{N-1} 2^{-k} \binom{N}{k}  
\frac{(-1)^{N-k-1}}{(N-k-1)!} (\beta + 1)_{N-k-1} 
\end{eqnarray*}
\noindent 
that reduces to the stated form.
\end{proof}

To obtain a hypergeometric form of the sum $S_{N}(n)$, observe that 
\begin{equation}
N(1 -N)_{\ell} = (-1)^{\ell} \frac{N!}{(N- \ell -1)!}
\end{equation}
and $(2)_{\ell} = ( \ell + 1)!$ give
\begin{equation}
(-1)^{\ell} \binom{N}{\ell+1} = N \frac{(1- N)_{\ell}}{(2)_{\ell}},
\end{equation}
\noindent
and the following result follows from Proposition \ref{propos-2}.

\begin{Prop}
The hypergeometric form of the sum $S_{N}(n)$ is given by
\begin{equation}
S_{N}(n) = \frac{(2n)!}{(2n-N)!} \beta^{2n-N+1} 2^{1-N} N  \, \, 
\pFq21{1-N, \quad1+\beta}{2}{2}.
\end{equation}
\end{Prop}

The final form of the sum $S_{N}(n)$ involves the Meixner-Pollaczek polynomials 
defined by 
\begin{equation}
P_{n}^{(\lambda)}(x;\phi) = \frac{(2 \lambda)_{n}}{n!} e^{\imath n \phi} 
\pFq21{-n, \quad \lambda + \imath x}{2 \lambda}{1 - e^{-2 \imath \phi}}. 
\end{equation}
\noindent
Choosing $\lambda = 1$ and $\phi = \pi/2$ gives the next result. 

\begin{theorem}
The sum $S_{N}(n)$ is given by
\begin{equation}
S_{N}(n) = \frac{(2n)!}{(2n-N)!} \frac{1}{( 2 \imath)^{N-1}} \beta^{2n-N+1} 
P_{N-1}^{(1)} \left( - \imath \beta; \frac{\pi}{2} \right). 
\label{sumN-hyper}
\end{equation}
\end{theorem}

Some examples are presented next. 

\begin{Example}
The Meixner-Pollaczek polynomial
\begin{equation}
P_{2}^{(1)} \left( x; \frac{\pi}{2} \right) = 2x^{2}-1
\end{equation}
\noindent
gives 
\begin{eqnarray*}
S_{3}(n) & = & \frac{(2n)!}{(2n-3)!} \times ( -1/4) \beta^{2n-2}(-2 \beta^{2} -1) \\
 & = & \frac{(2n)(2n-1)(2n-2)}{4} \left[ 2 \frac{B_{2n}}{2n} + \frac{B_{2n-2}}{2n-2} \right] \\
 & = & (2n-1)(n-1)B_{2n} + \tfrac{1}{2} n(2n-1)B_{2n-2},
 \end{eqnarray*}
 \noindent
 which coincides with  \cite[eq. (2.6)]{dilcher-1996a}. 
 \end{Example}
 
 \begin{Example}
 The Meixner-Pollaczek of degree $3$ is 
 \begin{equation}
 P_{3}^{(1)} \left(x; \frac{\pi}{2} \right) = \frac{4}{3} ( - 2x + x^{3} )
 \end{equation}
 \noindent 
 that produces 
 \begin{eqnarray*}
 S_{3}(n) & = & \frac{(2n)!}{(2n-4)!} \frac{1}{(2 \imath)^{3}} \beta^{2n-3} 
 \frac{4}{3} ( 2 \imath \beta + \imath \beta^{3}) \\
 & = & - \tfrac{1}{3}(2n-1)(n-1)(2n-3)B_{2n} - \tfrac{1}{3} (2n)(2n-1)(2n-3)B_{2n-2},
 \end{eqnarray*}
 \noindent
 which coincides with  \cite[eq. (2.7)]{dilcher-1996a}. 
 \end{Example}
 
 The next step is to establish a correspondence between the \textit{Dilcher 
 coefficients} $b_{k}^{(N)}$ in \eqref{dilcher-coeff} and the coefficients $p_{k}^{(n)}$ 
 in 
 \begin{equation}
 P_{n}^{(1)}(x;\pi/2) = \sum_{k=0}^{n}p_{k}^{(n)}x^{k}
 \end{equation}
 the Meixner-Pollaczek polynomials.  In particular, it is shown that 
 the recurrence \eqref{dilcher-rec} 
 is a consequence of the classical three terms recurrence for orthogonal 
 polynomials.

\begin{theorem}
The coefficients $b_{k}^{(N)}$ defined in \eqref{dilcher-coeff} and the coefficients 
$p_{k}^{(n)}$ are related by 
\begin{equation}
b_{k}^{(N)} = \frac{(-1)^{N-1-k}}{2^{N-1}} p_{N-1-2k}^{(N-1)}.
\label{rel-b-p}
\end{equation}
\noindent 
The recurrence relation \eqref{dilcher-rec} is equivalent to the three-terms recurrence 
\begin{equation}
(n+1) P_{n+1}^{(1)} \left( x; \frac{\pi}{2} \right) - 2x P_{n}^{(1)} \left( x; \frac{\pi}{2} \right) +
(n+1) P_{n-1}^{(1)}\left(x; \frac{\pi}{2} \right) = 0.
\label{three-rec}
\end{equation}
\noindent
satisfied by the Meixner-Pollaczek polynomials.
\end{theorem}
\begin{proof}
The Meixner-Pollaczek polynomials are orthogonal, hence they satisfy a three-terms 
recurrence. The specific form for this family in  \eqref{three-rec} appears 
in \cite[Chapter 18]{olver-2010a}. In terms of its coefficients $p_{k}^{(n)}$ this is 
expressed as 
\begin{equation}
(n+1)p_{k}^{(n+1)} - 2 p_{k-1}^{(n)} + (n+1)p_{k}^{(n-1)} = 0.
\label{rec-mp}
\end{equation}
\noindent
Comparing the two expressions for $S_{N}(n)$ in \eqref{dilcher-coeff} and \eqref{sumN-hyper} gives 
\eqref{rel-b-p}. This is equivalent to 
\begin{equation}
p_{\ell}^{(N-1)} = 2^{N-1}\imath^{N-1+\ell} b_{\tfrac{1}{2}(N-1-\ell)}^{(N)}.
\end{equation}
\noindent
Replacing in \eqref{rec-mp} and simplifying yields \eqref{dilcher-rec}.
\end{proof}

Theorem 2 in \cite{dilcher-1996a}, stated below,  may be proven along the same lines of the proof 
of Theorem \ref{the-stir1}. Details are omitted.

\begin{theorem}
If $2n \leq N-1$, then 
\begin{eqnarray}
S_{N}(n) & = &  (-1)^{n} \frac{(2n)! (N-2n-1)!}{2^{N-1}} p_{N-2n-1}^{(N-1)} \\
 & = & (-1)^{N-1} (2n)! (N-2n-1)! b_{n}^{(N)}. \nonumber 
\end{eqnarray}
\end{theorem}

\medskip

\noindent
{\bf Acknowledgments}. 
The work of the first author was partially funded by
$\text{NSF-DMS } 1112656$. \\


\bigskip


\begin{thebibliography}{1}

\bibitem{carlitz-1960a}
L.~Carlitz.
\newblock Note on {N}\"{o}rlund polynomial ${B}_{n}^{(z)}$.
\newblock {\em Proc. Amer. Math. Soc.}, 11:452--455, 1960.

\bibitem{dilcher-1996a}
K.~Dilcher.
\newblock Sums of products of {B}ernoulli numbers.
\newblock {\em Journal of Number Theory}, 60:23--41, 1996.

\bibitem{dixit-2014a}
A.~Dixit, V.~Moll, and C.~Vignat.
\newblock The {Z}agier modification of {B}ernoulli numbers and a polynomial
  extension. {P}art {I}.
\newblock {\em The Ramanujan Journal}, To appear, 2014.

\bibitem{lucase-1878a}
E.~Lucas.
\newblock Sur les congruences des nombres {E}uleriens et des coefficients
  diff\'{e}rentiels des fonctions trigonometriques, suivant un module premier.
\newblock {\em Bull. Soc. Math. France}, 6:49--54, 1878.

\bibitem{norlund-1924a}
N.~E. N\"{o}rlund.
\newblock {\em Vorlesungen \"{u}ber {D}ifferenzen-{R}echnung}.
\newblock Berlin, 1924.

\bibitem{olver-2010a}
F.~W.~J. Olver, D.~W. Lozier, R.~F. Boisvert, and C.~W. Clark, editors.
\newblock {\em {NIST} {H}andbook of {M}athematical {F}unctions}.
\newblock Cambridge {U}niversity {P}ress, 2010.

\bibitem{roman-1984a}
S.~Roman.
\newblock {\em The {U}mbral {C}alculus}.
\newblock Dover, New York, 1984.

\end{thebibliography}
\end{document}